\documentclass[12pt]{amsart}
\usepackage{amsmath,amssymb,amsthm,bm,longtable,color}

 \usepackage{a4wide}

\usepackage{latexsym,amssymb,color,pgf,tikz}
\usetikzlibrary{arrows,shapes,positioning}
\numberwithin{equation}{section}
\newtheorem{theorem}{Theorem}[section]

\newtheorem{lemma}[theorem]{Lemma}

\newtheorem{remark}[theorem]{Remark}
\newcommand{\SL}{\mathop{\mathrm{SL}}}

\newcommand{\PSL}{\mathop{\mathrm{PSL}}}

\newcommand{\alt}{\mathop{\mathrm{Alt}}}
\newcommand{\aut}{\mathop{\mathrm{Aut}}}

\newcommand{\sym}{\mathop{\mathrm{Sym}}}

\newcounter{claim}[theorem]

\title[Vertex-transitive graphs and local action $S_n$]{Vertex-transitive graphs with local action the symmetric group on ordered pairs}

\author{Luke Morgan}

\thanks{The author would like to thank the Isaac Newton Institute for Mathematical Sciences, Cambridge, for support and hospitality during the programme ``Groups, representations and applications: new perspectives'' where work on this paper was undertaken. This work was supported by EPSRC grant no EP/R014604/1. This work is supported in part by the Slovenian Research Agency (research program P1-0285 and research projects J1-1691, N1-0160, J1-2451).}

\address{Luke Morgan \\ University of Primorska, UP FAMNIT, Glagolja\v{s}ka 8, 6000 Koper, Slovenia, and University of Primorska, UP IAM, Muzejski trg 2,  6000 Koper, Slovenia.}
\email{luke.morgan@famnit.upr.si}


\begin{document}

\begin{abstract}
We consider a finite, connected and simple graph $\Gamma$ that admits a vertex-transitive group of automorphisms $G$.  Under the assumption that, for all $x \in V(\Gamma)$, the local action $G_x^{\Gamma(x)}$ is  the action of $\sym(n)$  on ordered pairs, we show that the group $G_x^{[3]}$, the pointwise stabiliser of a ball of radius three around $x$, is trivial.
\end{abstract}

\maketitle

\section{introduction}

In this paper, all graphs are finite, connected and simple and all groups are assumed to be finite. Let $L$ be a permutation group and let $\Gamma$ be a graph with $ G \leqslant \aut(\Gamma)$. For a vertex $x$ of $\Gamma$, the \emph{local action}  (at $x$) is $G_x^{\Gamma(x)}$ -- the group induced by $G_x$, the stabiliser of $x$ in $G$, on $\Gamma(x)$, the neighbours of $x$ in $\Gamma$. We   say that the pair $(\Gamma,G)$ is locally $L$ if $G$ acts transitively on the vertex set of $\Gamma$ and for each vertex $x$ of $\Gamma$  the local action $G_x^{\Gamma(x)}$  is permutationally isomorphic to $L$. The conjectures of Weiss \cite{WeissC}, Praeger \cite{PraegerC} and Poto\v{c}nik-Spiga-Verret \cite{psv} connect properties of the local action with global properties of the graph -- in terms of the number of automorphisms. For the purposes of this article, we phrase these conjectures using the following terminology (see \cite{verret}). A permutation group $L$ is called \emph{graph-restrictive} if there is a constant $c$ (depending on $L$) such that for each locally $L$ pair $(\Gamma,G)$ the equation $G_x^{[c]}=1$ holds. (Here the group $G_x^{[i]}$ is the point-wise stabiliser of a ball of radius $i$ around a vertex $x$.) The aforementioned conjectures say that primitive, quasiprimitive and semiprimitive groups, respectively, are graph-restrictive.  Moreover, Poto\v{c}nik, Spiga and Verret proved that a graph-restrictive group must be semiprimitive \cite[Theorem]{psv}, and therefore it makes sense to study this particular class of groups for these conjectures. When the conclusion of the conjectures hold, the group $G_x$ must act faithfully on the set of vertices at distance $c$ from $x$, and therefore has order bounded by $(d(d-1)^{c})!$, where $d$ is the valency of $\Gamma$ (also the degree of $L$).


In \cite[Proposition 14]{psv}, the open cases of the Poto\v{c}nik-Spiga-Verret Conjecture up to degree twelve are listed; these compromise two groups of degree nine (one a primitive group), a primitive group of degree ten and a semiprimitive group (that is not quasiprimitive) of degree twelve.  The groups of degree nine are known to be graph-restrictive by results of Spiga \cite{spigaffine} and Giudici, Morgan \cite{lukemichael3n2}. The primitive group of degree ten is $\sym(5)$ in its action on the set of ten unordered pairs  of elements of $\{1,2,3,4,5\}$, and therefore should be studied as part of that infinite family of groups (and within the context of the Weiss Conjecture).
   The next  case is the permutation group $L$ induced by $\sym(4)$ on the set of twelve ordered pairs of elements of $\{1,2,3,4\}$. This permutation group is not quasiprimitive, and so  is an interesting  test case of the validity of the Poto\v{c}nik-Spiga-Verret Conjecture.

Semiprimitive permutation groups were first studied by Bereczky and Mar\'{o}ti \cite{BM}. A transitive group is semiprimitive if each normal subgroup is transitive or semiregular. Recent investigations \cite{GiudiciMorgan} have revealed that the class of semiprimitive groups do admit a structure theory, albeit  not one as rigid as that of the primitive or quasiprimitive groups, which are described by their respective O'Nan-Scott type theorems \cite{LPSON,onQP}. The semiprimitive groups are first divided into three classes according to the  \emph{plinth} type. A plinth is a minimally transitive normal subgroup. A semiprimitive group having more than one plinth is known to be graph-restrictive by \cite{GiudiciMorgan}. The remaining two types have unique plinths and consist of those groups with regular plinth and those groups with non-regular plinth.  The class of semiprimitive groups a unique non-regular plinth  includes all almost simple primitive groups. The group $L$ defined above falls into the class of semiprimitive groups with a regular normal plinth (here $[L,L]\cong \alt(4)$ is regular).  Previous work on groups with regular normal plinth have focussed on the case that a plinth is nilpotent \cite{giudicimorganLRN}. 

In this article we show that the group $L$ is indeed graph-restrictive. The group $L$ is a member of a general family of semiprimitive permutation groups; those induced by $\sym(n)$ on the set of ordered pairs of distinct elements of $\{1,2,\ldots,n\}$. These groups, which have degree $n(n-1)$, are all found to be graph-restrictive by the following theorem.

\begin{theorem}
\label{thm: intro}
Let $n \geqslant 3$ be an integer and let $L \cong \sym(n)$ be the permutation group  induced by the action of $\sym(n)$ on the set of ordered pairs of distinct elements of $\{1,\ldots,n\}$. Let $(\Gamma,G)$ be a locally $L$ pair. If $n=3$ then $G_x^{[1]}=1$ holds, if $n=5$ or $n\geqslant 7$ then $G_x^{[2]}=1$ holds, and if $n=4$ or $6$ then $G_x^{[3]}=1$ holds.
\end{theorem}

The result above depends in a crucial way on so-called `pushing up' results that allow us to control normalisers of $p$-subgroups and on `failure of factorisation'. The relevance of these methods has long been known, since Weiss' result on locally affine graphs  \cite{weissaffine}. The most striking application is the recent verification of the affine case of the  Weiss Conjecture  by Spiga \cite{spigaffine}. Pushing up results may be applied to graph-restrictive problems if the point-stabiliser contains a Sylow subgroup for a specific prime. This fails for the class of groups we consider here; instead we leverage a conjugacy class of  subgroups that do contain a Sylow subgroup for the specific prime. The ideas in this article are therefore reasonably elementary; however we expect that the application of the pushing up results in this novel way may suggest new avenues for work on the conjectures mentioned at the beginning of this introduction.

In Section 2 we assemble the tools needed for the proof of Theorem~\ref{thm: intro}, which is then given in Section 3. In Section 4 we provide a construction which shows that the equation $G_x^{[3]}=1$ in Theorem~\ref{thm: intro} cannot be improved for $n=4$. 

\section{preliminaries}
We use the following notation. If $\Gamma$ is a graph and $x\in V(\Gamma)$, $\Gamma(x)$ denotes the neighbourhood in $\Gamma$ of $x$. If $G\leqslant \aut(\Gamma)$ and $i$ is a non-negative integer,
 $G_x^{[i]}$ denotes the pointwise stabiliser of vertices at distance at most $i$ from $x$ in $\Gamma$. If $x$ and $y$ are adjacent vertices of $\Gamma$, we may write  $G_{xy}^{[i]}$ for   $G_x^{[i]} \cap G_y^{[i]}$. For $i=0$, we suppress the superscripts.

The so-called `Hauptlemma' below  is used repeatedly both throughout this investigation, and in most papers on this problem.

\begin{lemma}[Hauptlemma]
If $K$ is a subgroup of $G_{xy}$  that is normal in $G_e$ and such that $N_{G_x}(K)$ is transitive on $\Gamma(x)$, then $K=1$.
\end{lemma}

Let $X$ be a $p$-group. We denote by $\Omega_1(Z(X))$ the subgroup of  $Z(X)$, the centre of $X$, generated by the elements of order $p$. We  use the following  definition of the Thompson subgroup: $J(X)$ is the subgroup of $X$ generated by the maximal (by order) elementary abelian subgroups of $X$. The group $J(X)$ is a characteristic subgroup of $X$, and it has the property that if $Y$ is a subgroup of $X$ with $J(X) \leqslant Y$, then $J(X) = J(Y)$. (See \cite[(32.1)]{Aschbacher}.)

For an arbitrary group $G$ and  $p$ a prime, $O_p(G)$ denotes the largest normal $p$-subgroup of $G$ and $O^p(G)$ denotes the smallest normal subgroup of $G$ such that $G/O^p(G)$ is a $p$-group. The generalised Fitting subgroup $F^*(G)$ is the product of $F(G)$, the largest normal nilpotent subgroup of $G$ and $E(G)$, the layer of $G$. The generalised Fitting subgroup has the property that $C_X(F^*(G))\leqslant F^*(G)$. We refer the reader to \cite[Chapter 11]{Aschbacher} for more details. 

The following is a Thompson-Wielandt style theorem proved for the semiprimitive case by Spiga.

\begin{theorem}[{\cite[Corollary 3]{SpigaTW}}]
\label{thm:tw}
Let $L$ be a semiprimitive group, $(\Gamma,G)$ be a locally $L$ pair and $\{x,y\}$ an edge of $\Gamma$. Then there is a prime $p$ such that   $G_{xy}^{[1]}$ is a $p$-group and  one of the following holds:
\begin{enumerate}
\item $G_{xy}^{[1]}=1$, or
\item $F^*(G_{xy}) = O_p(G_{xy})$.
\end{enumerate}
\end{theorem}

The pushing up result  that we need concerns the groups $\SL(2,2^n)$. Such a result was originally proved by Baumman \cite{baumann}. We cite below a generalisation due to Stellmacher that is more suited to our purpose.


\begin{theorem}[{\cite[Theorem 1]{Stellmacher}}]
\label{thm: Stell}
Let $M$ be a finite group, $p$ a prime and $S$ a Sylow $p$-subgroup of $M$ such that no non-trivial characteristic subgroup of $S$ is normal in $M$. Assume that
$$\overline{M}/\Phi(\overline{M}) \cong \PSL(2,p^f) \quad \text{for} \quad \overline{M} := M/O_p(M)$$
and set $V=[O_p(M),O^p(M)]$.
Then either $S$ is elementary abelian, or there exists an automorphism $\alpha \in \aut(S)$ such that 
$$L/V_0O_{p'}(L) \cong \SL(2,p^f) \quad \text{for} \quad L:=O^p(M)V^\alpha \quad \text{and}\quad V_0 = V(L \cap Z(M))$$
and one of the following holds:
\begin{itemize}
\item[(a)] $V \leqslant Z(O_p(M))$ and $V$ is a natural $\SL(2,2^f)$-module for $L/V_0O_{p'}(L)$.
\item[(b)] $V \leqslant Z(O_p(M))$, $p=2$ and $f>1$, and $V/ V \cap Z( M)$ is a natural $\SL(2,p^f)$-module for $L/V_0O_{p'}(L)$.
\item[(c)] $Z(V) \leqslant Z(O_p(M))$, $p\neq 2$, $\Phi(V) = V \cap Z(M)$ has order $p^f$, and $V/Z(V)$ and $Z(V)/\Phi(V)$ are natural $\SL(2,p^f)$-modules for $L/V_0O_{p'}(L)$.
\end{itemize}
\end{theorem}

The following result belongs to the theory of coprime action.

\begin{lemma}
\label{lem:coprime action}
Let $R$ be a group acting on a $p$-group $X$. Then \begin{itemize}
\item[(i)] for a  $q$-subgroup $S$ of $R$ with $q\neq p$,   $[X,S]=[X,S,S]$,
\item[(ii)] $[X,O^p(R),O^p(R)] = [X,O^p(R)]$.\end{itemize}
\end{lemma}
\begin{proof}
See \cite[(24.5)]{Aschbacher}.
\end{proof}


\begin{lemma}
\label{lem: approx}
Let $n\geqslant 4$ and let $L\cong \sym(n)$ be the permutation group induced by $\sym(n)$ on the set $\Omega$ of ordered pairs of distinct elements of $\{1,2,\ldots,n\}$. For $\omega \in \Omega$, there is a $N_L( L_\omega )$-conjugacy class of subgroups $\{T_1, T_2\}$  such that the following hold:
\begin{enumerate}
\item $L_\omega \leqslant T_i$;
\item $T_i\cong \sym(n-1)$;
\item $L = \langle N_L( L_ \omega), T_i \rangle$;
\item $L = \langle T_1, T_2   \rangle$.
\end{enumerate}
Furthermore, if there is a prime $p$ such that $L_\omega$ has a normal $p$-subgroup, then $L_\omega$ contains a Sylow $p$-subgroup of $T_i$ for $i=1,2$ and  $\langle O^p(T_1), O^p(T_2) \rangle$ is transitive on $\Omega$.
\end{lemma}
\begin{proof}
Let $\omega = (a,b)$, so that $L_\omega = L_a \cap L_b  \cong \sym(n-2)$. Observe that $L_\omega$ is contained in $L_a$, $L_b$ and $L_{\{a,b\}}$. Additionally, $N_L(L_\omega) = L_{\{a,b\}}$ so that $\{L_a^{N_L(L_\omega)} \} = \{ L_a, L_b\}$.
Thus (1)--(2) hold for $T_1=L_a$ and $T_2=L_b$. Since each $T_i$ is maximal in $L$, (3) and  (4) are immediate. For the final part, we must have $n\in \{4,5,6\}$, and it is an easy calculation to check $\langle O^p(T_1), O^p(T_2) \rangle$ is either $\alt(n)$ or $\sym(n)$.
\end{proof}

\section{Proof of the main theorem}

Let $n\geqslant 3$ be an integer and let  $L$  be the permutation group induced by the action of $\sym(n)$ on  $\Omega = \{(a,b) : 1 \leqslant a,b \leqslant n, a\neq b\}$.   Let $(\Gamma, G)$ be a locally $L$ pair and let $\{x,y\}$ be an edge of $\Gamma$. If $n=3$, then $L$ is regular, so $G_x^{[1]}=1$. Henceforth assume that $n\geqslant 4$ and that $G_{xy}^{[1]}\neq 1$. After applying  Theorem~\ref{thm:tw}  there is a prime $p$ such that $G_{xy}^{[1]}$,  $S_{xy}=F^*(G_{xy})$ and $Q_x=F^*(G_x^{[1]})$ are nontrivial $p$-groups. We set the following notation.

\begin{eqnarray*}
S_{xy} & = &  O_p(G_{xy})\\
Z_{xy} & = & \Omega_1(Z(S_{xy})) \\
Q_x & =&  O_p(G_x^{[1]})\\
Z_x & = & \Omega_1(Z(Q_x))
\end{eqnarray*}
and for $g\in G$ with $x^g=u$ and $y^g=v$, we define $S_{uv}=(S_{xy})^g$, $Q_u=(Q_x)^g$, etc. Since $Q_x$ and $S_{xy}$ are nontrivial, $Z_x$ and $Z_{xy}$ are nontrivial.

\begin{claim}
\label{claim: semireg cents}
The groups $C_{G_x}(Q_x)$ and $C_{G_x}(Z_x)$ are intransitive and semiregular on $\Gamma(x)$. 
\end{claim}

Since $G_{xy}^{[1]} \leqslant Q_x \cap Q_y$, the Hauptlemma shows that $C_{G_x}(Q_x)$ cannot be transitive on $\Gamma(x)$. Hence $C_{G_{xy}}(Q_x) = C_{G_x^{[1]} }(Q_x) \leqslant Q_x$. 
Now, since $Q_xQ_y \leqslant S_{xy}$, we see $Z_{xy}$ centralises $Q_x$, and so $Z_{xy} \leqslant C_{G_{xy}}(Q_x) = C_{G_x^{[1]}}(Q_x) \leqslant Q_x$. Thus $Z_{xy} \leqslant Z_x$. Since $Z_{xy}$ is nontrivial and normal in $G_e$,  the Hauptlemma shows $C_{G_x}(Z_x)$ must be intransitive on $\Gamma(x)$.

  Since $S_{xy}$ is normal in $G_e$ and  $Q_x$ is normal in $G_x$, the Hauptlemma shows that $S_{xy}\neq Q_x$. It follows that $S_{xy} \nleqslant G_x^{[1]}$ and hence  $S_{xy}^{\Gamma(x)}$ is a nontrivial normal $p$-subgroup of $G_{xy}^{\Gamma(x)} \cong \sym(n-2)$. Thus $(n,p)\in \{(4,2),(5,3),(6,2)\}$. (In particular, if $n\geqslant 7$ we have shown $G_{xy}^{[1]}=1$.)

Since $Q_x \neq Q_y$ (otherwise $Q_xQ_y=Q_x$ is normalised by $\langle G_x, G_e \rangle$), we have $(Q_xQ_y)^{\Gamma(x)}$ is a nontrivial normal $p$-subgroup of $G_{xy}^{\Gamma(x)}$.  Considering the various cases for $G_{xy}^{\Gamma(x)}$, it follows that $S_{xy} = Q_xQ_y$.

Let $R_x = \langle Q_xQ_y^{G_x} \rangle$. Now $R_x/Q_x$ is a central extension of $(R_x \cap G_x^{[1]})/Q_x$ (see \cite[Lemma 2.5]{spigaffine}) and since $Q_x = O_p(G_x^{[1]})$, $(R_x \cap G_x^{[1]} )/Q_x$ has order prime to $p$. Further, $R_x / (R_x \cap G_x^{[1]}) \cong R_x^{\Gamma(x)}$ is isomorphic to  $\sym(4)$, $\alt(5)$ or $\alt(6)$, if $n=4,5,6$, respectively.

\begin{claim}
\label{claim: p=2} 
We have $n=4$ or $6$ and $p=2$.
\end{claim}

Assume that $(n,p)=(5,3)$. Then $R_x/Q_x \cong \alt(5)$ or $\SL(2,5)$ and $R_x /C_{R_x}(Z_x) \cong \alt(5)$ or $\SL(2,5)$. Note that $J(Q_xQ_y) \leqslant Q_x$ means $J(Q_x) = J(Q_xQ_y)$ and it follows from the Hauptlemma that $Q_x =1$, a contradiction. Hence there is an elementary abelian subgroup of maximal order $A\leqslant Q_xQ_y$ with $A\nleqslant Q_x$. Now $C_{Z_x}(A)A$ is elementary abelian, so $|C_{Z_x}(A) A| \leqslant |A|$, so we have $C_{Z_x}(A) \leqslant A$, and so $C_{Z_x}(A) = Z_x \cap A$. Now, using \ref{claim: semireg cents}, $C_A(Z_x)= A \cap C_{Q_xQ_y}(Z_x)  \leqslant (Q_x Q_y )\cap G_x^{[1]} = Q_x$. Hence $C_A(Z_x) = A \cap Q_x$. Further, since $Z_xC_A(Z_x)$ is elementary abelian, we have
$$\frac{|Z_x| | C_A(Z_x) | }{ |Z_x \cap A|} = \frac{|Z_x| | C_A(Z_x) | }{|C_{Z_x}(A)|} = |Z_x C_A(Z_x)| \leqslant |A|$$
and so $|Z_x / C_{Z_x}(A)| \leqslant |A/C_A(Z_x)|$, which is to say that $A$ is an ``offender'' on $Z_x$. Note that $ |A/ A \cap Q_x|=3$.
If $W$ is a $R_x/C_{R_x}(Z_x)$ composition factor of $Z_x$, then $|W / C_{W}(A)| \leqslant |Z_x/C_{Z_x}(A)|$ which (by considering the irreducible $\mathrm{GF}(3)$ modules for $\alt(5)$ and $\SL(2,5)$) implies that $W$ must be the trivial module. Lemma~\ref{lem:coprime action} shows that $Z_x$ is centralised by $O^3(R_x)$.  On the other hand, $O^3(R_x)$ is transitive on $\Gamma(x)$, a contradiction to \ref{claim: semireg cents}. This proves \ref{claim: p=2} and so  we may now assume that $(n,p)=(4,2)$ or $(6,2)$.

By Lemma~\ref{lem: approx} there is a $N_L(L_\omega)$-conjugacy class of subgroups $\{T_1, T_2\}$ such that $T_i \cong \sym(n-1 )$ and $L_\omega \leqslant T_i$ for $i=1,2$. Let $R_i$ be the full preimage of $T_i$ in $G_x$. Note that the subgroups $R_1, R_2$  are a $N_{G_x}(G_{xy})$-conjugacy class.


Set $R_i^\circ = \langle (Q_xQ_y)^{R_i} \rangle$ and note that since $S_{xy}=O_p(G_{xy})=Q_xQ_y$ we have that $\{R_1^\circ,R_2^\circ\}$ is a $N_{G_x}(G_{xy})$-conjugacy class. Then $R_i^\circ$ is normal in $R_i$ and since $[R_i^\circ, G_x^{[1]} ] \leqslant Q_x$, we see that $R_i^\circ / Q_x $ is a central extension of $R_i^\circ / (R_i^\circ \cap G_x^{[1]})$ by $(R_i^\circ \cap G_x^{[1]})/ Q_x$. In particular, since $(R_i^\circ \cap G_x^{[1]})/ Q_x$ is abelian, and $Q_x = O_p(G_x)$, we have that $p \nmid |(R_i^\circ \cap G_x^{[1]})/ Q_x|$. It follows that
$ R_i^\circ / Q_x \cong \sym(3)$ if $n=4$ and $R_i^\circ/Q_x \cong \alt(5)$ if $n=6$.
In all cases we have that $Q_xQ_y$ is a Sylow $p$-subgroup of $R_i^\circ$ and that $O_p(R_i^\circ) = Q_x$. Note that $C_{R_i^\circ}(Q_x) \cap G_x^{[1]} \leqslant Q_x$. Suppose (for a contradiction) that $C_{R_i^\circ}(Q_x) \nleqslant G_x^{[1]}$, then $C_{G_x}(Q_x)G_x^{[1]}/G_x^{[1]}$ is a nontrivial normal subgroup of $G_x/G_x^{[1]}$. If $n=6$, it is immediate that $C_{G_x}(Q_x)$ must be transitive on $\Gamma(x)$ and if $n=4$, then the fact that $3$ divides the order of $C_{R_i^\circ}(Q_x)$ implies that $C_{G_x}(Q_x)$ is transitive on $\Gamma(x)$. This contradicts \ref{claim: semireg cents}. Hence we have:

\begin{claim}
For $i=1,2$,
\begin{itemize}
\item[(i)] $O_p(R_i^\circ) = Q_x$,
\item[(ii)] $C_{R_i^\circ}(Q_x) \leqslant Q_x$ and $R_i^\circ / Q_x \in \{ \sym(3),  \alt(5)\}$,
\item[(iii)] $Q_x Q_y$ is a Sylow $p$-subgroup of $R_i^\circ$.
\end{itemize}
\end{claim}

Suppose that $ C$ is a characteristic subgroup of $S_{xy}=Q_xQ_y$. Since $S_{xy}$ is characteristic in $G_{xy}$, it follows that $C$ is normal in $G_e$. If $C$ is normalised by $R_i^\circ$, then  $C$ is normalised by
$$\langle N_{G_x}(G_{xy}), R_i^\circ \rangle = \langle N_{G_x}(G_{xy}), R_1^\circ, R_2^\circ \rangle.$$
Lemma~\ref{lem: approx}(3) shows that  the image of this subgroup in $G_x/G_x^{[1]}$ is a transitive subgroup. Hence $C$ is normalised by a transitive subgroup of $G_x$ and so the Hauptlemma implies that  $C=1$. Thus:

\begin{claim} For $i=1,2$, no nontrivial characteristic subgroup of $Q_xQ_y$ is normal in $R_i^\circ$.
\end{claim}

By Claims (2) and (3) (using the isomorphisms $\sym(3)\cong \SL(2,2)$ and $\alt(5) \cong \SL(2,4)$), for $i=1,2$ we may apply Theorem~\ref{thm: Stell} to $R_i$, by setting $M=R_i$ and $Q_xQ_y= S$. Since $C_{R_i^\circ}(Q_x) \leqslant Q_x$,  $Q_xQ_y$ cannot be elementary abelian. Hence, since $p=2$, one of the outcomes (a) or (b) holds for $R_1^\circ$ and $R_2^\circ$ (independently). Let $V_i = [Q_x, R_i^\circ]$. Note that $C_{Q_x}(R_1^\circ)$ and $C_{Q_x}(R_2^\circ)$ are conjugate under $N_{G_x}(G_{xy})$. 

\begin{claim}
For $i=1,2$ we have  $[G_x^{[2]}, O^p(R_i^\circ)] \leqslant Z(R_i^\circ)$.
\end{claim}

Suppose for a contradiction that $[G_x^{[2]}, O^p(R_i^\circ)] \nleqslant Z(R_i^\circ)$ for some $i$. Since $G_x^{[2]}$ is normal in $G_x$, for $g\in G_x$ with $(R_1^\circ)^g = R_2^\circ$ we have 
$$ [G_x^{[2]}, O^p(R_1^\circ) ] \leqslant Z(R_1^\circ) \Leftrightarrow [G_x^{[2]}, O^p(R_1^\circ)]^g \leqslant Z(R_1^\circ)^g \Leftrightarrow [G_x^{[2]}, O^p(R_2^\circ)] \leqslant Z(R_2^\circ).$$
Hence, we may assume that $[G_x^{[2]}, O^p(R_i^\circ)] \nleqslant Z(R_i^\circ)$ for $i=1,2$.  Now  $[G_x^{[2]}, O^p(R_i^\circ)] \leqslant V_i$, and $V_i/(V_i \cap Z(R_i^\circ))$ is a simple module, so   $V_i = [G_x^{[2]}, O^p(R_i^\circ)] (V_i \cap Z(R_i^\circ))$. Let $U= \langle (G_{xy}^{[1]})^{R_i^\circ} \rangle$, and note $U \leqslant Q_x$. Then
$$[ U , O^p(R_i^\circ)] \leqslant [Q_x , O^p(R_i^\circ)] = V_i  =   [G_x^{[2]}, O^p(R_i^\circ) ](V_i \cap Z(R_i^\circ))$$
so 
$$[ U , O^p(R_i^\circ), O^p(R_i^\circ)] \leqslant  [[ G_x^{[2]}, O^p(R_i^\circ) ](V_i \cap Z(R_i^\circ)), O^p(R_i^\circ)] = [ G_x^{[2]}, O^p(R_i^\circ) ] \leqslant G_x^{[2]} \leqslant G_{xy}^{[1]}$$
By coprime action, $[ U , O^p(R_i^\circ), O^p(R_i^\circ)] =[ U, O^p(R_i^\circ)] $ this means $[G_{xy}^{[1]}, O^p(R_i^\circ)] \leqslant G_{xy}^{[1]}$ and so as above, $O^p(R_i^\circ)$ normalises $G_{xy}^{[1]}$ for $i=1,2$. Since  $\langle O^p(R_1^\circ), O^p(R_2^\circ)  \rangle$ is transitive on $\Gamma(x)$ by Lemma~\ref{lem: approx}, the Hauptlemma implies  $G_{xy}^{[1]}=1$, a contradiction.

\begin{claim}
\label{claim: cop action}
The group $G_{xy}^{[2]}$ is trivial.
\end{claim}

Coprime action  gives $[G_x^{[2]}, O^p(R_i^\circ)] = [G_x^{[2]},O^p(R_i^\circ), O^p(R_i^\circ)]$. By (3), for $i=1,2$ we have $[G_x^{[2]}, O^p(R_i^\circ),O^p(R_i^\circ)]  \leqslant [Z(R_i^\circ),O^p(R_i^\circ)]=1$. Since the statement holds for $i=1,2$, we have that $G_{xy}^{[2]}$ is normalised by $\langle O^p(R_1^\circ), O^p(R_2^\circ) \rangle$, and as above, the Hauptlemma implies $G_{xy}^{[2]}=1$.
%
%
%
%
This completes the proof of Theorem~\ref{thm: intro}.

 \section{examples}

In this section we give a construction that allows us to produce  examples of locally semiprimitive graphs from locally quasiprimitive graphs. We use the theory of amalgams. An amalgam is a triple of groups $(A,B,C)$ where $C$ is a specified subgroup of $A$ and of $B$. From a pair $(\Gamma,G)$ with $G \leqslant \aut(\Gamma)$, we get an amalgam $(G_x,G_e,G_{xy})$ for each edge $\{x,y\}$. If $G$ is vertex-transitive and $G_x^{\Gamma(x)}$ is transitive, then this amalgam is an invariant of the pair $(\Gamma,G)$ (so we may speak of ``the'' vertex-edge stabiliser amalgam of the pair). Conversely, amalgams give such pairs.  An amalgam $(A,B,C)$ is called faithful if the only subgroup of $C$ that is normal in both $A$ and $B$ is the trivial subgroup. If $(A,B,C)$ is a faithful amalgam, $|B:C|=2$  and $L$ is the permutation group induced by $A$ on the set of cosets of $B$ in $A$, then there exists a locally $L$ pair $(\Gamma,G)$ and the vertex-edge stabiliser amalgam of $(\Gamma,G)$ is $(A,B,C)$. For full details we refer the reader to \cite[Lemma 2.1]{MSV}.
 
Suppose that $L \leqslant \sym(\Omega)$ is a semiprimitive group with an intransitive semiregular normal subgroup $S$ and let  $\Delta$ be the set of $S$-orbits. Since $L$ is semiprimitive  the kernel of the action of $L$ on $\Delta$ is $S$ (see \cite[Lemma 3.1]{GiudiciMorgan}) and so $L^\Delta=L/S=:M$. Note that for $\omega \in \Omega$ and $\delta \in \Delta$, $L_\omega \cong M_\delta$, since $S$ is semiregular. Suppose now that $(\Sigma,H)$  is a locally $M$ pair and let $(H_x, H_e,H_{xy})$ be the vertex-edge stabiliser amalgam of $(\Sigma,H)$. We  construct a faithful amalgam $(G_x,G_e,G_{xy})$ such that the permutation group induced by $G_x$ on the set of cosets of $G_{xy}$ is $L$ and $|G_e:G_{xy}|=2$ and it will follow that there exists a locally $L$ pair $(\Gamma,G)$ with vertex-edge stabiliser amalgam $(G_x,G_e,G_{xy})$.

Let $A = L \times  H_x$. Then $X = S \times H_x^{[1]}$ is a normal subgroup of $A$. Further, 
$$A/X \cong L/S \times H_x/H_x^{[1]} = \{(Sa,H_x^{[1]}b) : a \in L ,  b\in H_x\}.$$
 Since $L/S = L^\Delta = M$ and $H_x/H_x^{[1]} \cong M$, there is a permutation isomorphism   $\phi :  H_x/H_x^{[1]} \rightarrow  L/S$. We define  $G_x$ to be a subgroup of $A$ containing $X$ such that 
 $$G_x/ X = \{ ( \phi(b), b) : b \in H_x/H_x^{[1]}\}.$$
 For $G_{xy}$ we look inside the subgroup $L_\omega \times H_{xy}$ of $A$. Now $H_x^{[1]}$ is normal in this group, and since $H_{xy}^{\Gamma(x)} \cong M_\delta \cong L_\omega$ we see that $(L_\omega \times H_{xy})/H_x^{[1]} \cong M_\delta \times M_\delta$. We  define  $G_{xy}$ to be the subgroup of $L_\omega \times H_{xy}$ containing $H_x^{[1]}$ such that 
 $$G_{xy} / H_x^{[1]} = \{ ( \phi( b) ,   b) : b\in H_{xy}/H_x^{[1]} \}.$$
 Since $\phi$ is an isomorphism, in fact $G_{xy} \cong H_{xy}$. Notice that $G_{xy} \leqslant G_x$ since   $G_{xy} \leqslant S G_{xy}$ and the elements of $S G_{xy}$  project to elements in $A/X$ that lie in $G_x/ X$. Set $G_e = H_e$. Since $G_{xy} \cong H_{xy}$, $G_{xy}$ is isomorphic to a subgroup of $G_e$ and $|G_e:G_{xy}|=2$. Thus 
 $$(G_x, G_e, G_{xy})$$
 is an amalgam. We compute $G_x^{[1]}:=core_{G_x}(G_{xy})$. Working modulo $S$, we see that the largest normal subgroup of $G_x$ contained in $G_{xy}$ must be contained in $S H_x^{[1]}$. Since $G_{xy} \cap S = 1$ and $H_x^{[1]} \leqslant G_{xy}$, we have that 
 $$H_x^{[1]} \leqslant core_{G_x}(G_{xy}) \leqslant  G_{xy} \cap (H_x^{[1]} S) = H_x^{[1]} (G_{xy} \cap S) = H_x^{[1]}.$$
Thus $G_x^{[1]} = H_x^{[1]}$ and so $G_{xy}^{[1]} = core_{G_e}(G_x^{[1]}) = core_{H_e}(H_x^{[1]}) = H_{xy}^{[1]}$. For $i\geqslant 2$ we define  inductively $G_x^{[i]} = core_{G_x}(G_{xy}^{[i-1]})$ and $G_{xy}^{[i]} = core_{G_e}(G_x^{[i]})$. Then, repeating these arguments, we have, for $i\geqslant 1$, $G_x^{[i]}=H_x^{[i]}$ and $G_{xy}^{[i]}=H_{xy}^{[i]}$. Since $(H_x,H_e,H_{xy})$ is a faithful amalgam, it follows that $(G_x,G_e,G_{xy})$ is a faithful amalgam. Furthermore, since the core in $G_x$ of $G_{xy}$ is $G_x^{[1]}=H_x^{[1]}$ and from   $A/H_x \cong L$ and $G_x \cap H_x = H_x^{[1]}$, it follows that $G_x/G_x^{[1]}$, the permutation group induced by the action of $G_x$ on the set of cosets of $G_{xy}$ is permutationally isomorphic to the action of $L$ on the cosets of $L_\omega$.  Thus there exists a locally $L$ pair $(\Gamma,G)$.
 
 Finally, we  apply the construction above to the group $L=S_4$ with $L_\omega = \langle (1,2) \rangle$ and $S = V_4$. Then $L/S \cong S_3$. This allows us to take any faithful amalgam $(H_x,H_e,H_{xy})$ with $H_x^{\Gamma(x)}  \cong S_3$ and produce an amalgam $(G_x,G_e,G_{xy})$ with local action $L$. The amalgams of locally $S_3$ pairs $(\Gamma,G)$ are classified by Djokovic-Miller \cite{DjokMiller}, and since there exist amalgams with $H_x^{[2]} \neq 1$ and $H_x^{[3]} = 1$, this shows that Theorem~\ref{thm: intro} is best possible for $n=4$.
 
 \begin{remark}
 The construction  above shows that  examples of locally semiprimitive graphs can be constructed from locally quasiprimitive examples. A question of relevance for the Poto\v{c}nik-Spiga-Verret Conjecture is whether all locally semiprimitive graphs arise in this way. In the construction, the intransitive semiregular normal subgroup $S$ is forced into the local action on the graph and also appears as a normal subgroup in the vertex-stabiliser. Such nice behaviour is perhaps too much to hope for.
\end{remark}


\begin{thebibliography}{99}
 \bibitem{Aschbacher} Aschbacher, M. Finite Group Theory. Cambridge University Press. (2000).
 
 \bibitem{baumann} Baumann, B. {\"{U}}ber Endliche Gruppen Mit Einer Zu {$L_2(2^n)$} Isomorphen Faktorgruppe.
 Proceedings of the American Mathematical Society, Vol. 74, No. 2, 215--222.
 

 \bibitem{BM} Bereczky, {\'A}ron and Mar{\'o}ti, Attila.  On groups with every normal subgroup transitive or semiregular. J. Algebra. 319(4) (2008), 1733--1751.
  

%


\bibitem{DjokMiller}   Djokovic, D.~Z.,  and Miller, G. L. Regular Groups of Automorphisms of Cubic Graphs Journal Of Combinatorial Theory, Series B 29, (1980) 195--230.
 


\bibitem{lukemichael3n2} Giudici, Michael and Morgan, Luke. A class of semiprimitive groups that are graph-restrictive. Bull. London Math. Soc. 46 (2014) 1226--1236.

\bibitem{giudicimorganLRN}
Giudici, Michael and  Morgan, Luke. On locally semiprimitive graphs and a theorem of Weiss, J. Alg. 427 (2015) 104--117.

\bibitem{GiudiciMorgan}
Giudici, Michael and  Morgan, Luke. A theory of semiprimitive groups, J. Alg.  503 (2018) 146--185.



\bibitem{LPSON}
Liebeck, Martin W., Praeger, Cheryl E.~ and Saxl, Jan.
On the O'Nan--Scott theorem for finite primitive permutation groups.
J. Austral. Math. Soc. (Series A) 44 (1988), 389--396.
 
 
 \bibitem{MSV} Morgan, Luke, Spiga, Pablo and Verret, Gabriel. On the order of Borel subgroups of group amalgams and an application to locally-transitive graphs. Journal of Algebra 434 (2015) 138--152.
      
\bibitem{psv} Poto\v{c}nik, Primoz, Spiga, Pablo and Verret, Gabriel. On graph-restrictive permutation
groups, J. Combin. Theory Ser. B 102 (2012), 820--831.

 
  \bibitem{onQP}Praeger, Cheryl E. An {O}'{N}an-{S}cott theorem for finite quasiprimitive
              permutation groups and an application to {$2$}-arc transitive
              graphs. J. London Math. Soc. (2), 47, (1993), 227--239.
 
 \bibitem{PraegerC}
Praeger, Cheryl E.~ Finite quasiprimitive group actions on graphs and designs, in: Young Gheel Baik, David L. Johnson, Ann Chi
Kim (Eds.), Groups -- Korea, de Gruyter, Berlin, New York, (2000), pp. 319--331.



%
%
%
%
  
  
  
  
\bibitem{SpigaTW} Spiga, Pablo. Two local conditions on the vertex stabiliser of arc-transitive graphs and their effect on the Sylow subgroups. J. Group Theory 15 no. 1 (2012), 23--35. 






\bibitem{spigaffine} Spiga, Pablo. An application of the Local C(G,T) Theorem to a conjecture of Weiss. 
Bull. London Math. Soc. 48 (2016) 12--18.


\bibitem{Stellmacher} Stellmacher, B. Pushing Up. Arch. Math., Vol. 46, (1986) 8--17.

%

\bibitem{verret}
Verret, G., On the order of arc-stabilisers in arc-transitive graphs. Bull. Aust. Math. Soc. 80 (2009), 498--505.


\bibitem{WeissC}
Weiss, R. s-transitive graphs. Algebraic methods in graph theory, Vol. I, II (Szeged, 1978), pp. 827--847, Colloq. Math. Soc. J\'{a}nos Bolyai, 25, North-Holland, Amsterdam-New York, 1981. 




\bibitem{weissaffine} Weiss, Richard.
An application of p-factorization methods to symmetric graphs, Math. Proc. Camb. Phil. Soc. (1979), 85, 43--48.
 
 \end{thebibliography}
\end{document}